\documentclass[reqno]{amsart}
\usepackage[margin=1.5in]{geometry}
\usepackage{palatino, mathpazo}
\usepackage{amsmath, amssymb, amsthm, mathrsfs}
\usepackage{tikz-cd}
\usepackage[colorlinks=true, allcolors=purple]{hyperref}

\newtheorem{thmx}{Theorem}

\newtheorem{theorem}{Theorem}[section]

\newtheorem{proposition}[theorem]{Proposition}
\newtheorem{corollary}[theorem]{Corollary}
\theoremstyle{definition}
\newtheorem{definition}[theorem]{Definition}
\newtheorem{remark}[theorem]{Remark}
\newtheorem{observation}[theorem]{Observation}

\newcommand{\bfem}[1]{\textbf{\textit{#1}}}
\newcommand{\LL}{\mathbb{L}}
\newcommand{\TT}{\mathbb{T}}
\newcommand{\Gm}{\mathbb{G}_m}
\newcommand{\cO}{\mathcal{O}}
\newcommand{\cL}{\mathcal{L}}
\newcommand{\cK}{\mathcal{K}}
\newcommand{\K}{\mathbb{K}}
\newcommand{\Map}{\mathrm{Map}}

\newcommand{\Spec}{\mathrm{Spec}}

\newcommand{\Hom}{\mathrm{Hom}}

\makeatletter
\g@addto@macro{\endabstract}{\@setabstract}
\newcommand{\authorfootnotes}{\renewcommand\thefootnote{\@fnsymbol\c@footnote}}%
\makeatother

\begin{document}

\title{AKSZ Construction for Shifted Contact Structures}

\author{Efe \.{I}zbudak}
\author{Kadri \.{I}lker Berktav}

\date{}

\begin{center}
  \LARGE
  AKSZ Construction for Shifted Contact Structures \par \bigskip

  \normalsize
  \authorfootnotes
  Efe \.{I}zbudak\footnote{efe.izbudak@studium.uni-hamburg.de}\textsuperscript{1}
  Kadri \.{I}lker Berktav\footnote{berktav@metu.edu.tr}\textsuperscript{2},
   \par \bigskip
  \textsuperscript{1}Department of Mathematics, Universit\"{a}t Hamburg\par 
  \textsuperscript{2}Department of Mathematics, METU\par
  \bigskip

\end{center}

\begin{abstract}
This paper establishes the AKSZ theorem for shifted contact structures and its applications. In brief, to resolve certain obstructions, we first define the quotient mapping stack as the quotient of the symplectified mapping space by the constant multiplicative group action. We then prove that if $X$ is an $n$-shifted contact derived Artin stack and $Y$ is an $\cO$-compact, $d$-oriented derived stack, the quotient mapping stack $$[\Map(Y, \widetilde{X})/\Gm]$$ admits an $(n-d)$-shifted contact structure. 

In addition, by formalizing the derived analogue of the graded contact AKSZ formalism, we also introduce the notion of \emph{weak shifted contact structures} in derived algebraic geometry and prove that under global trivialization of the contact line bundle, the unmodified mapping stack inherits a weak contact structure.

Extending this setup to spaces with boundary, we demonstrate that derived fillings naturally induce Legendrian morphisms between quotient mapping stacks, and that topological gluing of cobordisms evaluates to derived Legendrian intersections. Furthermore, we trace the transgression of the canonical shifted $1$-form to prove that our quotient mapping stacks satisfy the derived Classical Master Equation (CME).

As applications of our quotient mapping stack formalism, we define the derived analogues of specific topological field theories, including the Jacobi, Courant-Jacobi, and Loop Space Sigma Models. Next, we establish the contact analogue of the PTVV theorem for moduli spaces of perfect complexes. Finally, by composing this geometric construction with the perverse linearization in a companion paper, we elevate these moduli spaces to generate Cohomological Contact Extended Topological Field Theories.
\end{abstract}


\vspace{-.1in}

{\Small
\tableofcontents
}
\section{Introduction}

The AKSZ construction provides a formalism in derived algebraic geometry for generating shifted geometric structures on derived mapping stacks. Pantev, To\"{e}n, Vaqui\'{e}, and Vezzosi introduced this construction for shifted symplectic structures and proved that the derived mapping stack of a $d$-oriented stack into an $n$-shifted symplectic stack admits an $(n-d)$-shifted symplectic structure \cite[Theorem 2.1]{PTVV}. Calaque, Haugseng, and Scheimbauer generalized this construction to an extended topological field theory utilizing the higher category of Lagrangian correspondences \cite[Theorem 1.1]{CHS2025}. Tomi\'{c} adapted the AKSZ construction to shifted Poisson structures \cite[Theorem 3.4]{Tomic}.

Recently, Contreras, Martinez Alba, and Mehta developed an analogue of the AKSZ formalism for graded contact manifolds \cite[Theorem 4.1]{CMM}. Operating under the assumption of a globally trivializable contact line bundle, they demonstrated that the unmodified space of maps $\Map(\mathcal{N}, \mathcal{M})$ inherits a weak contact structure and provides solutions to the Classical Master Equation (CME) via the Jacobi bracket, yielding the Jacobi Sigma Model and Courant-Jacobi field theories in low dimensions.

In this paper, we transport the AKSZ-contact formalism into the setting of derived algebraic geometry and examine its applications concerning topological field theories. We first emphasize that a divergence occurs in the derived setting. Derived contact stacks carry non-trivial twisting line bundles. The fundamental weight 1 condition on the derived de Rham complex translates the Liouville scaling into an algebraic $\Gm$-action \cite[Section 2]{Calaque2019}. We prove that this topological twist algebraically obstructs the direct use of the unmodified mapping stack $\Map(Y, X)$. To resolve this and establish the derived contact AKSZ formalism, we define the quotient mapping stack as the quotient of the symplectified mapping space by the constant multiplicative group action.

We next promote our approach to the case of stacks with boundary structures. This is because, beyond the construction of the bulk space of fields, establishing a complete topological field theory requires geometric formulations of boundaries, locality, and integrability. To this end, we extend our formalism to stacks with a boundary and establish topological gluing, along with the transgression of homological data, to satisfy the Classical Master Equation.
As a further application of our approach, we also provide the contact PTVV theorem for moduli spaces of perfect complexes.
\medskip

Our main results are then summarized as follows.

\begin{thmx}[The Quotient Mapping Stack]
If $X$ is an $n$-shifted contact derived Artin stack and $Y$ is an $\cO$-compact, $d$-oriented derived stack, the quotient mapping stack $$X_Y = [\Map(Y, \widetilde{X})/\Gm]$$ admits an $(n-d)$-shifted contact structure (cf. Theorem \ref{thm:quotient_contact}).
\end{thmx}

We formulate the concept of \textit{weak shifted contact structures} in derived algebraic geometry to fully capture the derived analogue of the graded contact AKSZ formalism. By formalizing the graded geometry approach under the assumption of a globally trivializable contact line bundle, we identify the homological obstruction on the unmodified mapping stack.

\begin{thmx}[Weak Structures]
Under the assumption that $\cL_X \simeq \cO_X$ for an $n$-shifted derived Artin stack $X$, the unmodified mapping stack $\Map(Y, X)$ from an $\cO$-compact, $d$-oriented derived stack $Y$ inherits only a weak shifted contact structure due to the Tor-amplitude of the source space's orientation class (cf. Theorem \ref{prop:weak_contact}).
\end{thmx}

To construct a full topological field theory, the geometric evaluation of bulk spaces must naturally extend to spaces with boundary. We prove that derived fillings yield the appropriate morphisms in the contact category. In brief, we obtain:

\begin{thmx}[Legendrian Fillings]
For a derived filling $f \colon Y \to W$ of an $\cO$-compact, $d$-oriented derived stack $Y$ into an $n$-shifted contact derived Artin stack $X$, the natural restriction map between quotient mapping stacks defines a Legendrian morphism $X_W \to X_Y$ (cf. Theorem \ref{thm:filling_legendrian}).
\end{thmx}

Furthermore, this geometric assignment naturally extends to the composition of boundary conditions, allowing us to formalize topological gluing in the derived contact setting. In this regard, we prove:

\begin{thmx}[Topological Gluing]
Given two derived fillings $f_1 \colon Y \to W_1$ and $f_2 \colon Y \to W_2$ of an $\cO$-compact, $d$-oriented derived stack $Y$, the quotient mapping stack into an $n$-shifted contact derived Artin stack $X$ out of the homotopy pushout $W = W_1 \coprod_Y W_2$ is canonically equivalent to the derived intersection of the Legendrian boundary conditions (cf. Theorem \ref{thm:top glu}).
\end{thmx}

Beyond the topological structure of boundaries and gluing, constructing a physical field theory requires recovering the homological data that governs the space of fields. We show that our mapping stack formalism successfully manifests the transgression of the contact form. More precisely, we establish:

\begin{thmx}[Transgression of the Contact Form]
Given an $\cO$-compact, $d$-oriented derived stack $Y$ and an $n$-shifted contact derived Artin stack $X$, the integration functional transgresses the canonical $n$-shifted 1-form on the symplectified space $\widetilde{X}$ into a weight 1, $(n-d)$-shifted 1-form on $\widetilde{M}=\mathsf{Map}(Y,\widetilde{X})$ that descends to define the canonical shifted contact 1-form on the quotient mapping stack $X_Y$ (cf. Theorem \ref{thm: transgression of contact}).
\end{thmx}

With the canonical shifted 1-form established on the quotient mapping stack, we complete the geometric realization of the Classical Master Equation and prove:

\begin{thmx}[Derived CME and Sigma Models]
The quotient mapping stack recovers the derived analogues of transgressed contact forms, and the derived Classical Master Equation is governed by the equivalence of perfect complexes induced by the derived de Rham differential (cf. Proposition \ref{prop:derived_cme}). 

As applications, evaluating this geometry provides the derived moduli spaces of fields for the Jacobi ($n=1, d=2$), Courant-Jacobi ($n=2, d=3$), and Loop Space ($n=1, d=1$) Sigma Models. See Section \ref{sec:apps}.
\end{thmx}

Next, beyond specific low-dimensional examples, our formalism universally maps topological cobordisms of arbitrary dimension to non-linear Legendrian spans. To elevate these mapping stacks into fully Extended Topological Field Theories (ETFTs) and extract numerical invariants, we post-compose our geometric construction with the perverse linearization developed in the companion paper \cite[Theorem C]{Izbudak_Monodromic}, leading to:

\begin{thmx}[Linearized ETFTs]
The quotient mapping stack formalism maps topological cobordisms to Legendrian spans, generating a geometric contact ETFT. Post-composing this with the perverse linearization of \cite{Izbudak_Monodromic}, which is conditional on the monodromic refinement of the symplectic Joyce conjecture assumed there, generates Cohomological Contact ETFTs (cf. Theorem \ref{thm:linear_etft}).
\end{thmx}
\vspace{1in}

As a further application of our quotient mapping stack formalism, we establish the contact analogue of the PTVV theorem for moduli spaces of perfect complexes given in \cite[Theorem 0.1]{PTVV} and provide a complete characterization when the source stack $Y$ is Calabi-Yau, De Rham, or  Betti. In brief, we prove:
\begin{thmx}[Contact PTVV and Calabi-Yau Spaces] 
Let $Y$ be an $\cO$-compact derived stack. The quotient mapping stack 
$$[\Map(Y, T^*[2]\mathrm{Perf}^\circ)/\Gm]$$ 
canonically admits a strict $(2-d)$-shifted contact structure via transgression if and only if $Y$ is a Calabi-Yau derived stack (or Calabi-Yau category) of dimension $d$. See Proposition \ref{prop: contact Calabi-Yau}.
\end{thmx} It follows that when $Y$  is a smooth and proper Calabi-Yau variety of dimension $d$, the quotient mapping stack gives rise to the \textit{derived moduli stack of projective Higgs perfect
complexes on $Y$}; when  $Y = Z_{DR}$ is the de Rham stack of a smooth and proper variety $Z$ of dimension $d$, the quotient mapping stack evaluates to the \emph{projectivized cotangent stack of perfect complexes with flat connections} on $Z$, denoted $\mathbb{P}(T^*[2(1-d)]\mathrm{Perf}_{DR}(Z))$, which admits a canonical $2(1-d)$-shifted contact structure; when $Y = M_B$ is the Betti stack associated with a compact oriented topological manifold $M$ of dimension $d$, the quotient mapping stack evaluates to the \emph{projectivized cotangent stack of perfect complexes of local systems} on $M$, denoted $\mathbb{P}(T^*[2-d]\mathrm{Loc}(M, \mathrm{Perf}))$, which admits a canonical $(2-d)$-shifted contact structure. Proposition \ref{prop: contact Calabi-Yau}, with some follow-up discussion, details each case.
\medskip

\paragraph{\bfem{Organization of the paper}}
Section \ref{sec:preliminaries} reviews the foundations of derived algebraic geometry and shifted geometric structures. Section \ref{sec:contact_aksz} identifies the algebraic obstruction on the standard mapping stack and proves the existence of the contact structure on the quotient mapping stack. Section \ref{sec:weak_contact} analyzes the unmodified stack under trivial line bundle assumptions. Section \ref{sec:fillings} establishes that derived fillings induce Legendrian morphisms between quotient mapping stacks. Section \ref{sec:cme} verifies the derived Classical Master Equation. Section \ref{sec:apps} evaluates the AKSZ construction for specific dimensions to define derived contact Sigma Models and constructs Linearized Extended Topological Field Theories, and proves the contact analogue of the PTVV theorem for moduli spaces of perfect complexes. \vspace{0.1in}

\paragraph{\bfem{Conventions and notations}.} Throughout the paper, $ \mathbb{K} $ will be an algebraically closed field of characteristic zero. All cdgas will be graded in nonpositive degrees and  over $\mathbb{K}.$ We always consider $\K$-schemes/stacks, and we assume that all classical $ \K $-schemes are \emph{locally of finite type}, and that all derived $ \K $-schemes/stacks $ {X} $ are  \emph{locally finitely presented.}

\section{Preliminaries}\label{sec:preliminaries}

\subsection{Derived schemes and stacks} Let us first review the basics of derived algebraic geometry. The punchline is that the study of (affine) derived schemes is formally dual to
the study of the so-called \textit{commutative differential graded algebras},
providing an algebraic framework for derived geometry. These algebras generalize commutative rings by encoding relations and syzygies into the differential
structure.  More precisely, we have:
\begin{definition}[Commutative Differential Graded Algebra]
  A \bfem{commutative differential graded algebra} (cdga) over $\K$
  is a graded $\K$-algebra $A = \bigoplus_{i \le 0} A^i$ concentrated
  in non-positive cohomological degrees, equipped with a differential
  $d: A^i \to
  A^{i+1}$ satisfying $d^2 = 0$ and the graded Leibniz rule.
  The $\infty$-category of such algebras is denoted by
  $\mathrm{cdga}_\K^{\leq 0}$.
\end{definition}

We now present the underlying (higher) spaces of interest in this paper and outline their essential properties, along with several useful remarks.

\begin{definition}[Affine Derived Schemes]
  The $\infty$-category of \bfem{affine derived schemes} over $\K$,
  denoted $\mathbf{dAff}_\K$, is defined as the opposite
  $\infty$-category $(\mathrm{cdga}_\K^{\leq 0})^{\mathrm{op}}$. For
  any $A \in \mathrm{cdga}_\K^{\leq 0}$, the corresponding affine
  derived scheme is denoted $\Spec(A)$ \cite[Definition 2.2.2.4]{ToenHAG}.
\end{definition}

\begin{definition}[Derived Stacks]
  A \bfem{derived stack} over $\K$ is a functor $X \colon
  \mathrm{cdga}_\K^{\leq 0} \to \mathcal{S}$ taking values in the
  $\infty$-category of spaces (Kan complexes), such that $X$
  satisfies exact hyperdescent with respect to the \'{e}tale
  topology. The $\infty$-category of derived stacks, denoted
  $\mathbf{dSt}_\K$, is constructed as the left Bousfield
  localization of the $\infty$-category of presheaves
  $\mathrm{Fun}(\mathrm{cdga}_\K^{\leq 0}, \mathcal{S})$ along
  \'{e}tale hypercovers \cite[Definition 2.2.2.14]{ToenHAG}.
\end{definition}
In other words, a derived stack $ {X} $ is  a  functor  $$ {X}:  \mathrm{cdga}_{\K}^{\leq 0} \rightarrow  \mathcal{S}, \ \ A\mapsto X(A) \simeq \mathrm{Map}_{\mathbf{dSt}_\K}(\Spec(A), {X}),$$ satisfying a descent condition.
For more details, we refer to \cite{ToenHAG}. 
It should then be noted that any affine derived $\mathbb{K}$-scheme $X$ can be corepresented by a cdga $A$, where $\Spec(A)$ is given as the functor \[ \Spec (A): B\in \mathrm{cdga}_\K^{\leq 0}  \longmapsto Hom_{\mathrm{cdga}_{\mathbb{K}}^{\leq 0}} (A,B) \in \mathcal{S}.\]

The concept of a derived stack may sometimes be too wild for some purposes because a typical $A$-point $p\in X(A)$ is represented by a general morphism $p: \Spec A \to X$ in $\mathbf{dSt}_\K$. When we ask for additional conditions on those morphisms, such as being a Zariski open embedding or being smooth (of some relative dimension), the resulting space may admit a tractable description in terms of affine covers, leading to the following special kind of derived stacks.

\begin{definition}[Derived Schemes]
  \label{def:derived_schemes}
  An object $X$ in the $\infty$-category of derived stacks $\mathbf{dSt}_\K$ is a \bfem{derived scheme} over $\K$ if it can be covered by Zariski open affine derived schemes $Y \subset X$. The $\infty$-category of derived schemes, denoted $\mathbf{dSch}_\K$, is the full subcategory of $\mathbf{dSt}_\K$ spanned by the derived schemes.
\end{definition}

\begin{definition}[Smooth Morphisms of Derived Stacks]
  A morphism of derived stacks $f: X \to Y$ is \bfem{smooth} if it
  is locally of finite presentation and its relative cotangent
  complex $\LL_{X/Y}$ is a perfect complex of Tor-amplitude $[0, 0]$,
  meaning it acts homotopically as a locally free sheaf
  \cite[Section 2.2.5]{ToenHAG}.
\end{definition}

\begin{definition}[Derived Artin Stacks]
  A \bfem{derived Artin stack} is a derived stack
  which is $m$-geometric for some integer $m$, for the \'{e}tale
  topology, and for the class of smooth maps \cite[Definition
  1.3.1]{PTVV}. 
\end{definition}
As mentioned before, all derived $ \K $-schemes/stacks $ {X} $ are assumed to be \textit{locally finitely presented.} We now recall the deformation theory of these spaces and the shifted geometric structures they support.

\subsection{Shifted symplectic and contact structures} Let $L_{qcoh}(X)$ denote the \bfem{stable $\infty$-category of quasi-coherent complexes} on a derived stack $X$. The deformation theory of $X$ is defined by the \bfem{cotangent complex} $\LL_{X} \in L_{qcoh}(X)$, which is the quasi-coherent complex of derived K\"{a}hler differentials. The \bfem{tangent complex} $\TT_{X}$ is defined as the derived $\cO_X$-linear dual complex $\TT_X \simeq \mathbb{R}\underline{\mathcal{H}om}(\LL_X, \cO_X)$.

By \cite[Definition 1.12]{PTVV}, the \bfem{space of closed $p$-forms of degree $n$} on $X$ is defined via the mapping space $\mathrm{Map}_{\epsilon-\mathbf{dg}_\K^{gr}}(\K, \mathbf{DR}(X)[n+p](p))$ in the $\infty$-category of graded mixed $\K$-complexes. An underlying $p$-form of degree $n$ induces a class in the cohomology group $H^{n}(X, \wedge^{p} \LL_{X})$. Denote the spaces of such by $\mathcal{A}^{p,cl}_\K(X,n)$ and $\mathcal{A}^{p}_\K(X,n)$, respectively.

The space $\mathcal{A}^{p,cl}(X,n)$ is usually more complicated than the space $\mathcal{A}^p(X,n)$ even if we make standard geometric assumptions on $X$. Thanks to \cite[Prop. 1.14]{PTVV}, we have the following identification for the space $\mathcal{A}^p(X,n)$:
\begin{equation*}
    \mathcal{A}^p(X,n)\simeq {\rm Map}_{L_{qcoh}(X)}(\mathcal{O}_X, \wedge^p\mathbb{L}_X[n]). 
\end{equation*}
 Thus, any 2-form of degree $n$  induces a morphism $\mathcal{O}_X \rightarrow  \wedge^2\mathbb{L}_X[n]$ in $D_{qcoh}(X)$, and hence, by duality, a morphism $\mathbb{T}_X \wedge \mathbb{T}_X \rightarrow \mathcal{O}_X[n]$. By adjunction, this gives the induced morphism $\Theta_{\omega} : \mathbb{T}_X \rightarrow \mathbb{L}_X [n]$. Furthermore, there is a natural projection morphism $\mathcal{A}^{p,cl}(X,n) \to \mathcal{A}^p(X,n)$ from closed shifted forms to the underlying forms.

An \bfem{$n$-shifted symplectic structure} on $X$ is then defined to be a closed 2-form $\omega$ of degree $n$ whose underlying 2-form induces an equivalence $\Theta_\omega \colon \TT_X \xrightarrow{\sim} \LL_X[n] $ in $L_{qcoh}(X)$. 

A map $f \colon L \to X$ into a symplectic stack is \bfem{isotropic} if it is equipped with a null-homotopy $h \colon f^*\omega \simeq 0$. The isotropic structure $h$ on $f \colon L \to X$ induces a canonical morphism $\Theta_h \colon \TT_L \to \LL_{L/X}[n-1]$. Such structure is called \bfem{Lagrangian} if $\Theta_h$ is an equivalence.

Following \cite[Definition 3.1]{IzbudakBerktav2026}, an \bfem{$n$-shifted contact structure} on $X$ consists of a twisting line bundle $\cL$, an $n$-shifted twisted $1$-form $\alpha \in \pi_0 \mathbb{R}\Gamma(X, \LL_X \otimes^{\mathbb{L}}_{\cO_X} \cL[n])$, and a non-degenerate contact distribution defined as the homotopy fiber
$$
\cK \simeq \mathrm{fib}(\TT_X \xrightarrow{\alpha^\vee} \cL[n]).
$$
A morphism $f \colon L \to X$ into a contact stack is \bfem{Legendrian} if it is equipped with an isotropic null-homotopy $f^*\alpha \simeq 0$ that induces a stable fiber sequence
\begin{equation*}
\cO_L \to \LL_{L/X} \otimes^{\mathbb{L}}_{\cO_L} f^*\cL[n-1] \to \TT_L \quad \text{ in }\  L_{qcoh}(L).
\end{equation*}

The \bfem{derived symplectification} of an $n$-shifted contact stack $X$ is the total space of the principal $\Gm$-bundle $p \colon \widetilde{X} \to X$ associated with $\cL$. 
The $\Gm$-action on a derived stack induces a fiber grading on the derived de Rham complex over $B\Gm$ \cite[Section 3]{Calaque2019}. A closed $p$-form $\omega$ on $\widetilde{X}$ has \bfem{weight $w$} if it defines a point in the graded mapping space
$$
    \mathcal{A}^{p, \mathrm{cl}, \{w\}}(\widetilde{X}, n) \simeq \mathrm{Map}_{\epsilon-\mathbf{dg}_{\K}^{gr}}(\K, \mathbf{DR}(\widetilde{X})[n+p](p)\{w\}),
$$
where $\epsilon-\mathbf{dg}_{\K}^{gr}$ denotes the $\infty$-category of graded mixed complexes, and $\{w\}$ denotes the $w$-th shift in the fiber grading. The $n$-shifted symplectic form $\omega_{\widetilde{X}}$ on $\widetilde{X}$ carries a weight of 1. That observation plays a central role in manifesting the descending of geometric data. In that respect, the derived transversality theorem has been developed in \cite[Theorem 5.1]{IzbudakBerktav2026}, stating that the quotient of an $n$-shifted symplectic stack by a weight 1 $\Gm$-action defines an $n$-shifted contact structure on the base. That is, taking the quotient of a shifted symplectic stack by a weight 1 $\Gm$-action descends the symplectic data to contact data. More precisely, we have:
\begin{theorem}[Derived Analogue of Classical Transversality]  \label{thm:A}
Given an $n$-shifted symplectic derived Artin stack $(\widetilde{X},
  \omega_{\widetilde{X}})$
  equipped with a $\Gm$-action of weight 1,
  the stack quotient $[\widetilde{X} / \Gm]$ inherits an
  $n$-shifted contact structure.  
\end{theorem}
\vspace{.1in}

\paragraph{\bf Compactness, orientation, and more. } To develop the AKSZ construction, which relies on integration over mapping spaces, it is necessary to first adapt some topological concepts. Specifically, we have: 

\begin{definition}
    \begin{enumerate}\itemsep=.1in
        \item A derived stack $Y$ is \bfem{$\cO$-compact} if $\cO_Y$ is a compact object in $D_{qcoh}(Y)$ and for any perfect complex $E$, the derived global sections $C(Y, E)$ is perfect. 

\item An \bfem{$\cO$-orientation of dimension $d$} is a map $[Y] \colon C(Y, \cO_Y) \to \K[-d]$ that induces non-degenerate pairings.

\item For a morphism $f \colon Y \to W$ between $\cO$-compact derived stacks, a \bfem{boundary structure} is a path $\gamma$ from $f_*[Y]$ to $0$ in $\Map(C(W, \cO_W), \K[-d])$. A boundary structure is \bfem{non-degenerate} if it induces an equivalence $C(W, E^\vee) \xrightarrow{\sim} C(W, E)^\perp$ for any perfect complex $E$, where the orthogonal complement is the fiber of the pairing against $f^*E^\vee$. 

\item A \bfem{derived filling} is an $\cO$-orientation on $Y$ equipped with a non-degenerate boundary structure on $f \colon Y \to W$.
    \end{enumerate}
\end{definition}

\subsection{{The AKSZ construction in the derived symplectic setup}} 

With the categorical notions of orientation and boundary structures established, we can review the integration procedure they support. The following construction is well-known in differential geometry and provides, together with \cite[Theorem 2.5]{PTVV}, heuristics for our constructions in the upcoming sections.

Let $M$ be a compact $C^{\infty}$-manifold of dimension $m$, $N$ be a $C^{\infty}$-manifold, and $Map(M,N)$ the Fr\'echet manifold of $C^{\infty}$-maps from $M$ to $N$. The following diagram $$
\begin{tikzcd}[row sep=3em, column sep=3em]
& M \times \Map(M, N) \arrow[dl, "pr_M"'] \arrow[dr, "ev"] & \\
M & & N
\end{tikzcd}
$$
induces a map on differential forms $$\Omega_{M}^{p} \times \Omega_{N}^{q} \rightarrow \Omega_{Map_{C^{\infty}}(M,N)}^{p+q-m} : (\alpha, \beta) \mapsto \int_{M} pr^{*}_{M}\alpha \wedge ev^{*} \beta, $$ 
where $\int_{M}$ denotes \bfem{integration along the fiber} and $pr^*_M(\alpha) \wedge ev^{*}\beta \, \in \Omega_{M \times Map(M,N)}^{p+q}$. The so-called \textit{AKSZ construction (or transgression operation)}  asserts that if, in particular, $(N,\omega)$ is symplectic, then we consider the form $pr^*_M({\rm id}) \wedge ev^{*}\omega \, \in \Omega_{M \times Map_{C^{\infty}}(M,N)}^{0+2}$ such that the evaluation  $$\int_M pr^*_M({\rm id}) \wedge ev^{*}\omega \ \in \Omega_{Map(M,N)}^{2-m}$$ defines a symplectic form on $Map_{C^{\infty}}(M,N).$ Introducing suitable counterparts, this idea
can be translated into the derived context.

In the setting of derived algebraic geometry, one needs a replacement for Poincar\'e duality and for the notion of an orientation. The first one is given by Serre duality (in the more general context of Calabi-Yau categories), while the second one will be that of $\mathcal{O}$\emph{-orientation}. The notion of $\mathcal{O}$-orientation will allow for a quasi-coherent variant of the concept \emph{integration along the fiber} \cite{PTVV}.

Recall from \cite{PTVV} that for any derived stack $X$ and any $\mathcal{O}$-compact derived stack $Y$, there is a natural 
a commutative square of graded complexes
\[
\begin{tikzcd}
NC^{w}(X\times Y) \arrow[r, "\kappa_{X,Y}"] \arrow[d] 
  & NC^{w}(X) \otimes_{k} C(Y,\mathcal{O}_{Y}) \arrow[d] \\
DR(X\times Y) \arrow[r, "\kappa_{X,Y}"'] 
  & DR(X) \otimes_{k} C(Y,\mathcal{O}_{Y}),
\end{tikzcd}
\]
where the vertical morphisms are the projections $NC^{w} \longrightarrow  {DR}$ and the horizontal ones are induced by PTVV’s (affine) Künneth formula and then by projecting to the function part.  

We then have the following formal definition:
\begin{definition}\cite[Definition 2.3]{PTVV}
Let $X$ and $Y$ be derived stacks, with $Y$ being $\mathcal{O}$-compact, and 
let $\eta : C(Y,\mathcal{O}_{Y}) \longrightarrow k[d]$ be a morphism
for some integer $d$.
The \bfem{integration map along $\eta$} is the morphism 
\[
\int_{\eta} : 
\begin{tikzcd}[column sep=large]
DR(X\times Y) \arrow[r, "\kappa_{X,Y}"] 
  & DR(X)\otimes_{k} C(Y,\mathcal{O}_{Y}) \arrow[r, "id\otimes \eta"] 
  & DR(X)[d].
\end{tikzcd}
\]
This map commutes with all the structure (along with the corresponding morphisms in terms of $NC^{w}(-)$ on the level of mixed graded complexes).
\end{definition}

We are now in a position to outline the AKSZ procedure: let $X$ be a derived Artin stack  equipped with
an $n$-shifted symplectic form $\omega \in \mathsf{Symp}(X,n)$, 
$Y$ be an $\mathcal{O}$-compact derived stack equipped with an $\mathcal{O}$-orientation $[Y] : C(Y,\mathcal{O}_{Y}) \longrightarrow k[-d]$ of degree
$d$, and denote by \(
ev : Y\times \mathsf{Map}(Y,X) \longrightarrow X
\) the evaluation morphism. 

Note that since  
\(
\omega \in \mathsf{Symp}(X,n) \subset \mathcal{A}^{2,cl}(X,n),
\) the underlying 2-form corresponds to a morphism of graded complexes
\[
\omega : k[2-n](2) \longrightarrow NC^{w}(X).
\]The integration along the orientation $[Y]$ then yields the composition:
\[
\int_{[Y]}\omega :=  \int_{[Y]} \circ  \ ev^* \circ \omega : k[2-n](2) \longrightarrow NC^{w}(\mathsf{Map}(Y,X))[-d],
\] 
which is, by definition, a closed $2$-form of degree $(n-d)$ on $\mathsf{Map}(Y,X)$, and hence
\[
\int_{[Y]}\omega \in \mathcal{A}^{2,cl}(\mathsf{Map}(Y,X),n-d).
\]

The non-degeneracy of such a form
follows directly from the non-degeneracy of the orientation, in the special case of a symplectic pairing. For details, we refer to \cite[Theorem 2.5]{PTVV}. In brief, the discussion above proves:
\begin{theorem}[PTVV's AKSZ construction] 
The derived mapping stack $\mathsf{Map}(Y,X)$ of a $d$-oriented stack $Y$ into an $n$-shifted symplectic stack $X$ admits an $(n-d)$-shifted symplectic structure.    
\end{theorem}

In what follows, we aim to establish the contact counterpart of that result, using the equivariant quotient approach introduced in \cite{IzbudakBerktav2026}.

\section{Contact AKSZ Construction}\label{sec:contact_aksz}

While the symplectic AKSZ procedure operates directly on the derived mapping stack, extending this formalism to the contact setting requires geometric modifications. We first identify the homological obstruction on the standard mapping stack and then construct the quotient mapping stack to resolve it.

\subsection{Obstructions}

Let $Y$ be an $\cO$-compact derived stack equipped with an $\cO$-orientation of dimension $d$. We investigate the geometric structures on mapping spaces into an $n$-shifted contact derived Artin stack $X$.

The standard mapping stack $\Map(Y, X)$ does not admit a contact structure. A contact structure on $\Map(Y, X)$ requires a twisting line bundle. Let $\cL_X$ be the contact line bundle on $X$. The required line bundle must be constructed via the correspondence defined by the projection and evaluation morphisms
$$
\begin{tikzcd}[row sep=3em, column sep=3em]
& Y \times \Map(Y, X) \arrow[dl, "p_2"'] \arrow[dr, "ev"] & \\
\Map(Y, X) & & X
\end{tikzcd}
$$\\

The derived pushforward complex is defined as $\mathbb{R}p_{2*} ev^* \cL_X$. The fiber of this derived pushforward over a map $f \in \Map(Y, X)$ is the cohomology complex $\mathbb{R}\Gamma(Y, f^*\cL_X)$. For a $d$-oriented derived stack $Y$, this perfect complex has amplitude $[0, d]$. A contact structure requires a global twisting line bundle, which must be a perfect complex of amplitude $[0, 0]$ and constant rank 1. Since the derived pushforward yields a complex with non-trivial higher cohomology groups, it fails to be a line bundle. As such, the mapping stack $\Map(Y, X)$ cannot admit a contact structure.

We provide an explicit counterexample to demonstrate that this obstruction is an intrinsic cohomological property of the setup, and not a fluctuation across connected components.

\begin{observation}[Betti Stacks] 
    Let $Y$ be the Betti stack $M_B$ of a compact oriented $d$-manifold where $d \ge 1$. Evaluate the derived pushforward over a constant map $f \colon M_B \to X$ targeting a single point in the contact stack. The pullback $f^*\cL_X$ is isomorphic to the constant sheaf. The derived pushforward $\mathbb{R}\Gamma(M_B, f^*\cL_X)$ computes the standard singular cohomology $C^*(M, \K)$. Since $M$ is a compact, oriented $d$-manifold, its top cohomology group $H^d(M, \K)$ is non-zero by Poincar\'{e} Duality. Consequently, the pushforward complex has an absolute Tor-amplitude of $[0,d]$. The presence of non-vanishing higher cohomology violates the $[0,0]$ amplitude requirement, thus it is impossible to define a line bundle globally over the component of constant maps.
\end{observation}

One might attempt to bypass this topological obstruction using formal localization. Tomi\'{c} establishes in \cite{Tomic} that an $n$-shifted Poisson structure on $X$ is equivalent to an $(n+1)$-shifted Lagrangian thickening $X \to X^{symp}$ \cite[Theorem 4.1]{Tomic}. The AKSZ construction proceeds by applying the mapping stack functor to this thickening and taking the formal completion. 

For a shifted contact structure, the analogous approach requires lifting the geometry to an $(n+1)$-shifted Legendrian thickening. The mapping stack functor produces a morphism $\Map(Y, X) \to \Map(Y, X^{symp})$. The domain of this morphism is the standard mapping stack $\Map(Y, X)$. The standard mapping stack fails to support the required twisting line bundle due to the $[0, d]$ amplitude of the derived pushforward. Formal completion along this morphism cannot recover a contact structure since the base space lacks the necessary global line bundle. The geometric obstruction occurs before any formal neighborhood is analyzed. The lifting and descent mechanism over the de Rham prestack fails for contact structures.

This algebraic obstruction aligns with classical contact topology. In Symplectic Field Theory, moduli spaces of holomorphic curves in a contact manifold $M$ are not constructed from the mapping space $\Map(Y, M)$. Maps are taken into the symplectification $\mathbb{R} \times M$ \cite[Section 1]{EGH}, and the moduli space is defined as the quotient $[\Map(Y, \mathbb{R} \times M) / \mathbb{R}]$, where the group $\mathbb{R}$ acts by translation along the Liouville vector field. This translation corresponds to the scaling of the symplectic form. The quotient mapping stack provides the derived algebraic generalization of this moduli space. 

\subsection{Contact AKSZ via The Quotient Mapping Stack}
Motivated by the topological analogy above, we propose to resolve the derived obstruction by formalizing the quotient mapping stack as
$$
    X_Y = [\Map(Y, \widetilde{X})/\Gm],
$$ 
where $\widetilde{X} \to X$ is a principal $\Gm$-bundle equipped with a weight 1 $n$-shifted symplectic structure $\omega_{\widetilde{X}}$, that is, the derived symplectification of $X$.
 Here, taking the quotient by $\Gm$ constructs a principal $\Gm$-bundle over $X_Y$ that defines a twisting line bundle $\cO(1)$ on $X_Y$. Often denoting the derived mapping stack $\Map(Y, \widetilde{X})$ by $\widetilde{M}$, we more precisely have:

\begin{theorem}\label{thm:quotient_contact}
If $X$ is an $n$-shifted contact derived Artin stack and $Y$ is an $\cO$-compact, $d$-oriented derived stack, the quotient mapping stack $X_Y = [\Map(Y, \widetilde{X})/\Gm]$ admits an $(n-d)$-shifted contact structure.
\end{theorem}
\begin{proof}
By \cite[Theorem 4.2]{IzbudakBerktav2026}, the $n$-shifted contact structure on $X$ defines a principal $\Gm$-bundle $p \colon \widetilde{X} \to X$ equipped with a weight 1 $n$-shifted symplectic structure $\omega_{\widetilde{X}}$. We apply the derived mapping stack functor to this principal bundle, yielding the derived stack $\widetilde{M} = \Map(Y, \widetilde{X})$. By \cite[Theorem 2.5]{PTVV}, $\widetilde{M}$ admits an $(n-d)$-shifted symplectic structure $\omega_{\widetilde{M}}$ obtained via transgression along the evaluation morphism $ev \colon Y \times \widetilde{M} \to \widetilde{X}$. The transgressed symplectic form is given by $\omega_{\widetilde{M}} = \int_{[Y]} ev^* \omega_{\widetilde{X}}$.

The $\Gm$-action on $\widetilde{X}$ induces an action of the mapping group stack $\mathcal{G} = \Map(Y, \Gm)$ on $\widetilde{M}$. The constant maps from $Y$ to $\Gm$ define an embedding of derived group stacks $\Gm \hookrightarrow \mathcal{G}$. The restriction of the $\mathcal{G}$-action to this subgroup endows $\widetilde{M}$ with a pointwise $\Gm$-action. This subgroup acts on maps $f \colon Y \to \widetilde{X}$ by global scalar multiplication on the target space fibers. 

The symplectic form $\omega_{\widetilde{M}}$ must be an eigenform of weight 1 with respect to the induced $\Gm$-action. The transgression functional $\int_{[Y]}$ evaluates over the intrinsic orientation class of the derived stack $Y$. The fundamental class of $Y$ is invariant under the constant $\Gm$-action. Since $\omega_{\widetilde{X}}$ has weight 1 on $\widetilde{X}$, the pullback $ev^* \omega_{\widetilde{X}}$ has weight 1 on $Y \times \widetilde{M}$. Because the $\Gm$-action operates on the target fibers of the evaluation map, it commutes with the linear integration functional over the $\cO$-orientation. Integration over the invariant fundamental class of $Y$ preserves the scaling weight. Thus, $\omega_{\widetilde{M}}$ is a weight 1 $(n-d)$-shifted symplectic form.

We then conclude that the derived mapping stack $\widetilde{M} = \Map(Y, \widetilde{X})$ carries an $(n-d)$-shifted symplectic structure equipped with a weight 1 $\Gm$-action. From \cite{IzbudakBerktav2026} the quotient of a shifted symplectic stack by a weight 1 $\Gm$-action descends the symplectic structure to a contact structure on the quotient stack. This process is governed by the $\Gm$-equivariant descent square mapping the symplectified evaluation space to the contact base. 

The derived transversality theorem \cite[Theorem 5.1]{IzbudakBerktav2026} guarantees that this descent transfers the geometric data via the principal bundles. Therefore, the $\Gm$-equivariant descent endows the quotient mapping stack $X_Y = [\Map(Y, \widetilde{X})/\Gm]$ with an $(n-d)$-shifted contact structure.
\end{proof}

We then have the following immediate corollary encoding the relation between the operations of derived symplectification and taking the mapping stack.

\begin{corollary}
The operations of derived symplectification and taking the mapping stack commute. Specifically, the derived symplectification of the quotient mapping stack $X_Y$ is canonically equivalent to the derived mapping stack into the symplectification of $X$
$$
    \widetilde{X_Y} \simeq \Map(Y, \widetilde{X}).
$$
\end{corollary}
\begin{proof}
By definition, the derived symplectification $\widetilde{X_Y}$ is the total space of the principal $\Gm$-bundle associated with the contact structure on $X_Y$. The quotient mapping stack is defined as $X_Y = [\Map(Y, \widetilde{X})/\Gm]$. The canonical projection $\Map(Y, \widetilde{X}) \to X_Y$ is the presentation of this quotient as a principal $\Gm$-bundle. In the proof of the preceding theorem, the $(n-d)$-shifted contact structure on $X_Y$ is constructed by descending the weight 1 transgressed symplectic form $\omega_{\widetilde{M}}$ along this bundle. Therefore, $\Map(Y, \widetilde{X})$ satisfies the universal property of the derived symplectification of $X_Y$.
\end{proof}

\subsection{Weak Contact Structures on the Unmodified Stack}\label{sec:weak_contact}

Although the quotient mapping stack provides the proper setup for ``strong'' contact structures, the unmodified stack can still capture a relaxed version of this geometry under specific assumptions. In order to relate our constructions to recent developments in graded contact geometry, we introduce the following definition.

\begin{definition}
A \bfem{weak $n$-shifted contact structure} on a derived stack $M$ consists of a twisting line bundle $\cL_M$ and an $n$-shifted 1-form $\alpha_M$, such that the induced distribution $\cK_M \simeq \mathrm{fib}(\TT_M \xrightarrow{\alpha^\vee} \cL_M[n])$ defines a morphism (via $d_{dR}\alpha_M$) $$\Phi \colon \cK_M \to \cK_M^\vee \otimes \cL_M[n]$$  that is not an equivalence. 

With this terminology in hand, when we say \bfem{strong} in place of weak, we actually mean genuine shifted contact structures in the sense described in Preliminaries.
\end{definition}

In the graded AKSZ-contact formalism \cite[Section 3]{CMM}, the space of fields is constructed on the unmodified mapping space $\Map(Y, X)$. This is achieved by imposing a restrictive topological assumption. The contact line bundle on the target is globally trivializable ($\cL_X \simeq \cO_X$). We mirror this construction in the derived setup to formalize why it yields a weak structure.

\begin{theorem}\label{prop:weak_contact}
Let $X$ be an $n$-shifted contact derived Artin stack such that its contact line bundle is globally trivializable ($\cL_X \simeq \cO_X$). For any $\cO$-compact, $d$-oriented derived stack $Y$, the unmodified mapping stack $M = \Map(Y, X)$ admits a weak $(n-d)$-shifted contact structure.\end{theorem}
\begin{proof}
Because $\cL_X \simeq \cO_X$, the line bundle obstruction is bypassed. The derived pushforward $\mathbb{R}p_{2*} ev^* \cO_X \simeq \cO_M \otimes C(Y, \cO_Y)$ globally trivializes down to $\cO_M$ by utilizing the constant functions, allowing the definition of a global 1-form. The canonical shifted $1$-form $\alpha_X \in \pi_0 \mathbb{R}\Gamma(X, \LL_X[n])$ transgresses via the integration functional $\int_{[Y]}$ to a global $(n-d)$-shifted 1-form $\alpha_M \in \pi_0 \mathbb{R}\Gamma(M, \LL_M[n-d])$.

By base change, the tangent complex evaluates as $\TT_M \simeq \mathbb{R}p_{2*} ev^* \TT_X$. The canonical $1$-form $\alpha_M$ is obtained via the composition
$$ \TT_M \xrightarrow{\mathbb{R}p_{2*} ev^* \alpha_X^\vee} \mathbb{R}p_{2*} ev^* \cO_X[n] \xrightarrow{\int_{[Y]}} \cO_M[n-d]. $$
The derived contact distribution $\cK_M$ is the homotopy fiber of this composition. By the pasting law for exact triangles, $\cK_M$ fits into a fiber sequence
$$ \mathbb{R}p_{2*} ev^* \cK_X \to \cK_M \to \mathrm{fib}\left(\int_{[Y]}\right)[n]. $$
Because the integration map $\int_{[Y]}$ annihilates the cohomology of $Y$ below degree $d$, its homotopy fiber is non-zero. Thus, $\cK_M$ is not equivalent to $\mathbb{R}p_{2*} ev^* \cK_X$. 

The strong contact structure on $X$ yields an equivalence $\cK_X \xrightarrow{\sim} \cK_X^\vee[n]$. Pushing this equivalence forward and applying derived relative Serre duality along the $\cO$-orientation $[Y] \colon C(Y, \cO_Y) \to \K[-d]$ induces a perfect pairing on $\mathbb{R}p_{2*} ev^* \cK_X$. However, the induced pairing morphism on $M$
$$ \Phi \colon \cK_M \to \cK_M^\vee[n-d] $$
factors through the integration map. The non-trivial fiber of the integration map induces a degeneracy in $\Phi$, preventing it from being an equivalence. This establishes that $M$ supports a weak contact structure, recovering the results of \cite[Section 3]{CMM} in the derived setting.

\end{proof}

\subsection{Derived Fillings and Legendrian Morphisms} \label{sec:fillings}

Returning to the quotient mapping stack, we extend the geometric construction to source spaces with boundary. In shifted symplectic geometry, restricting mapping stacks along morphisms equipped with boundary structures yields Lagrangian morphisms. The quotient mapping stack formulation recovers the analogue for contact geometry. More precisely, we have:

\begin{theorem}\label{thm:filling_legendrian}
Let $f \colon Y \to W$ be a derived filling of an $\cO$-compact, $d$-oriented derived stack $Y$. For any $n$-shifted contact derived Artin stack $X$, the relative mapping stack defined by the restriction morphism $\rho \colon X_W \to X_Y$ between the quotient mapping stacks is equipped with an $(n-d)$-shifted Legendrian structure.
\end{theorem}
\begin{proof}
For any symplectic stack, derived fillings induce Lagrangian morphisms on the associated mapping stacks. The restriction map between the derived mapping stacks into the derived symplectification $$\widetilde{\rho} \colon \Map(W, \widetilde{X}) \to \Map(Y, \widetilde{X})$$ forms a Lagrangian morphism with respect to the transgressed $(n-d)$-shifted symplectic form on $\Map(Y, \widetilde{X})$.

The restriction map $\widetilde{\rho}$ is $\Gm$-equivariant because the $\Gm$-action defined by point-wise multiplication on the target fibers commutes with the domain inclusion $Y \to W$. Because the Lagrangian structure on $\widetilde{\rho}$ is $\Gm$-equivariant, it descends along the principal $\Gm$-bundle quotients to define a Legendrian morphism between the respective contact base spaces $X_W \to X_Y$, which directly follows from \cite[Theorem 6.2]{IzbudakBerktav2026}.

\end{proof}

Having established that individual derived fillings yield Legendrian morphisms, we can naturally extend this theorem to the composition of boundary conditions, thereby manifesting topological gluing.

\begin{theorem}[Topological Gluing]\label{thm:top glu}
Let $f_1 \colon Y \to W_1$ and $f_2 \colon Y \to W_2$ be two derived fillings of an $\cO$-compact, $d$-oriented derived stack $Y$. For an $n$-shifted contact derived Artin stack $X$, the quotient mapping stack out of the homotopy pushout $W = W_1 \coprod_Y W_2$ is canonically equivalent to the derived intersection of the Legendrian boundary conditions
$$
    X_W \simeq X_{W_1} \times_{X_Y}^h X_{W_2}.
$$
\end{theorem}
\begin{proof}
The derived mapping stack functor maps colimits in the source category to limits in the target category. The mapping stack into the symplectification $\Map(W, \widetilde{X})$ evaluates to the homotopy pullback $\Map(W_1, \widetilde{X}) \times^h_{\Map(Y, \widetilde{X})} \Map(W_2, \widetilde{X})$. 

Because the quotient functor to stacks over $B\Gm$ is an equivalence of $\infty$-categories, taking the $\Gm$-equivariant stack quotient commutes with the homotopy pullback. Evaluating the stack quotient yields the derived fiber product of the quotient mapping spaces over $X_Y$. By Theorem \ref{thm:filling_legendrian}, the morphisms into $X_Y$ are Legendrian, identifying $X_W$ as the derived Legendrian intersection.
\end{proof}

\begin{remark}
This gluing theorem formalizes the mechanism of \textit{derived Legendrian surgery}. In the $(\infty, 2)$-category of Legendrian spans \cite[Definition 4.1]{IzbudakBerktav2}, the composition of two 1-morphisms is defined via the derived fiber product over the shared boundary condition. The theorem establishes that the topological gluing of cobordisms on the base space $W_1 \coprod_Y W_2$ maps functorially to the geometric composition of the corresponding Legendrian spans $X_{W_1} \times_{X_Y}^h X_{W_2}$ in the contact moduli space.
\end{remark}

\section{Transgression and the Derived Classical Master Equation}\label{sec:cme}

Beyond the existence of the geometric space, the mapping stack must recover the homological data governing the physical field theory. In the graded AKSZ-contact formalism \cite[Section 5]{CMM}, the space of fields inherits a generalized contact Hamiltonian satisfying the Classical Master Equation (CME) via the Jacobi bracket $\{S, S\}_J = 0$. This section will build the analogous result in our framework. Throughout this section, we assume $X$ is an $n$-shifted contact derived Artin stack and $Y$ is an $\cO$-compact, $d$-oriented derived stack. We will denote by $\widetilde{M}$ the derived mapping stack $\Map(Y, \widetilde{X})$ for ease of writing.

\begin{theorem}[Transgression of the Contact Form]\label{thm: transgression of contact}
Let $E_{\widetilde{X}} \colon \cO_{\widetilde{X}} \to \TT_{\widetilde{X}}$ be the fundamental vector field of the $\Gm$-action, which determines the canonical $n$-shifted 1-form $\alpha_{\widetilde{X}} = \iota_{E_{\widetilde{X}}} \omega_{\widetilde{X}}$ on the symplectification. The integration functional transgresses $\alpha_{\widetilde{X}}$ into a weight 1, $(n-d)$-shifted 1-form $\alpha_{\widetilde{M}}$ on $\widetilde{M}$ that descends to define the canonical shifted contact 1-form on the quotient mapping stack $X_Y$.
\end{theorem}
\begin{proof}
Let $E_{\widetilde{M}} \colon \cO_{\widetilde{M}} \to \TT_{\widetilde{M}}$ denote the fundamental vector field of the induced pointwise $\Gm$-action on the mapping stack. Because the action operates on the target fibers of the evaluation map, $E_{\widetilde{M}}$ corresponds to the pullback of $E_{\widetilde{X}}$. 

We compute the canonical 1-form $\alpha_{\widetilde{M}}$ by taking the interior product of $E_{\widetilde{M}}$ with the transgressed symplectic form. Since the integration functional operates on the domain factors, it commutes with the contraction operator on the target space
$$
    \alpha_{\widetilde{M}} = \iota_{E_{\widetilde{M}}} \left( \int_{[Y]} ev^* \omega_{\widetilde{X}} \right) \simeq \int_{[Y]} ev^* \left( \iota_{E_{\widetilde{X}}} \omega_{\widetilde{X}} \right) = \int_{[Y]} ev^* \alpha_{\widetilde{X}}
$$
Because $Y$ carries the trivial $\Gm$-action, the integration functional preserves the $\Gm$-weights. Since $\alpha_{\widetilde{X}}$ is a weight 1 form, $\alpha_{\widetilde{M}}$ is a weight 1 form of degree $n-d$. 

By the local theory of shifted contact structures, the canonical 1-form on the symplectification admits the local description $\alpha_{\widetilde{X}} \simeq f \cdot p^*\alpha_X$, where $f \in \cO_{\widetilde{X}}^\times$ represents the scaling action along the fibers of the principal $\Gm$-bundle $p \colon \widetilde{X} \to X$. 

Because $\alpha_{\widetilde{M}}$ is a weight 1 eigenform and carries this identical scaling property along the target fibers of the mapping stack, it descends along the principal $\Gm$-bundle $\pi \colon \widetilde{M} \to X_Y$. The presence of the scaling factor $f$ geometrically dictates that the descended form does not map to functions, but rather to sections of the line bundle associated to the torsor. Thus, it descends to a global section twisted by the canonical line bundle $\cO(1)$ on the quotient
$$
    \alpha_{X_Y} \in \pi_0 \mathbb{R}\Gamma(X_Y, \LL_{X_Y} \otimes^{\mathbb{L}}_{\cO_{X_Y}} \cO(1)[n-d])
$$
This establishes the derived algebraic analogue of the weak contact form transgression $\breve{\alpha}$ utilized in the graded setting \cite[Section 5.1]{CMM}.
\end{proof}

With the canonical shifted 1-form established on the quotient mapping stack, completing the geometric realization of the Classical Master Equation requires verifying that this form induces a non-degenerate distribution.

\begin{proposition}[The Derived Classical Master Equation]\label{prop:derived_cme}
The quotient mapping stack $X_Y$ satisfies the geometric requirements of the derived Classical Master Equation. The underlying contact distribution is non-degenerate, yielding the equivalence of perfect complexes
$$
    \cK_{X_Y} \xrightarrow{\sim} \cK_{X_Y}^\vee \otimes^{\mathbb{L}}_{\cO_{X_Y}} \cO(1)[n-d]
$$
induced by the derived de Rham differential $d_{dR}\alpha_{X_Y}$.
\end{proposition}
\begin{proof}
The distribution $\cK_{X_Y}$ is defined as the homotopy fiber $\mathrm{fib}(\TT_{X_Y} \xrightarrow{\alpha_{X_Y}^\vee} \cO(1)[n-d])$. We verify the non-degeneracy by pulling the data back along the principal $\Gm$-bundle $\pi \colon \widetilde{M} \to X_Y$. 

On the symplectification, the derived orthogonal complement $E_{\widetilde{M}}^\perp$ is defined as the homotopy fiber of the contraction morphism $\TT_{\widetilde{M}} \xrightarrow{\iota_E \omega} \cO_{\widetilde{M}}[n-d]$. Because the fundamental vector field is isotropic ($\iota_E \iota_E \omega_{\widetilde{M}} \simeq 0$), there is a canonical lift defining a morphism $E_{\widetilde{M}} \colon \cO_{\widetilde{M}} \to E_{\widetilde{M}}^\perp$.

By derived base change along the torsor projection, the pulled-back tangent complex $\mathbb{L}\pi^*\TT_{X_Y}$ is the homotopy cofiber of $E_{\widetilde{M}} \colon \cO_{\widetilde{M}} \to \TT_{\widetilde{M}}$. Because $\alpha_{\widetilde{M}}$ descends to the base, the contraction map factors through this cofiber. By the pasting law for exact triangles, this identifies the pulled-back contact distribution $\mathbb{L}\pi^*\cK_{X_Y}$ as the homotopy cofiber of the isotropic inclusion $\cO_{\widetilde{M}} \to E_{\widetilde{M}}^\perp$.

This realizes $\mathbb{L}\pi^*\cK_{X_Y}$ as the derived symplectic reduction of $\widetilde{M}$ along the $\Gm$-action. By the PTVV mapping stack theorem \cite[Theorem 2.5]{PTVV}, the transgressed form $\omega_{\widetilde{M}}$ is a strict $(n-d)$-shifted symplectic structure. In the stable $\infty$-category of perfect complexes, the quotient of an orthogonal complement by its isotropic generator inherits a non-degenerate pairing. Thus, $\omega_{\widetilde{M}}$ induces the equivalence
$$
    \mathbb{L}\pi^*\cK_{X_Y} \xrightarrow{\sim} (\mathbb{L}\pi^*\cK_{X_Y})^\vee [n-d]
$$
in $L_{qcoh}(\widetilde{M})$. This equivalence is $\Gm$-equivariant and carries weight 1. Applying derived $\Gm$-equivariant descent along the torsor projection $\pi$ transfers this equivalence to the base space, twisting the target by the canonical line bundle $\cO(1)$. This establishes the non-degenerate equivalence $\cK_{X_Y} \xrightarrow{\sim} \cK_{X_Y}^\vee \otimes^{\mathbb{L}}_{\cO_{X_Y}} \cO(1)[n-d]$ in $L_{qcoh}(X_Y)$. The integrability of this geometry is guaranteed by the closure of the forms in the derived de Rham prestack, confirming that the moduli space satisfies the geometric CME.

\end{proof}

\section{Applications: Derived Contact Sigma Models}\label{sec:apps}

Having verified that the quotient mapping stack satisfies the geometric Classical Master Equation, the resulting moduli spaces possess the necessary homological data to encode the shifted BV formulation of topological field theories. We apply this setup to define the derived analogues of specific models and geometric integration constructions by selecting particular values for the geometric shift $n$ of the contact derived Artin stack $X$ and the orientation dimension $d$ of the $\cO$-compact derived stack $Y$. We again mean by $\widetilde{M}$ the derived mapping stack $\Map(Y, \widetilde{X})$.

\subsection{The Derived Jacobi Sigma Model (\texorpdfstring{$n=1, d=2$})}

Setting the geometric shift to $n=1$ recovers the target space geometry of Jacobi manifolds. A classical Jacobi manifold consists of a line bundle $\cL$ over a manifold $M$ equipped with a Lie bracket on its sections satisfying a generalized Leibniz rule. This structure is equivalent to a homogeneous Poisson structure on the principal $\Gm$-bundle associated to $\cL$.
 
In the derived setting, a derived Jacobi structure on a derived stack $X$ with line bundle $\cL$ is defined as a 1-shifted Poisson structure on the principal $\Gm$-bundle $p \colon \widetilde{X} \to X$ that is homogeneous of weight $-1$ with respect to the $\Gm$-action. Assuming this 1-shifted Poisson structure is non-degenerate, \cite[Theorem 4.1]{Tomic} shows that it induces a quasi-isomorphism between the tangent and shifted cotangent complexes, defining a 1-shifted symplectic form $\omega_{\widetilde{X}}$. The weight $-1$ condition on the Poisson bivector dualizes to a weight $1$ condition on the corresponding symplectic form $\omega_{\widetilde{X}}$. Thus, the derived symplectification $\widetilde{X}$ is a 1-shifted symplectic stack equipped with a weight 1 $\Gm$-action. By derived equivariant descent, the stack quotient $X \simeq [\widetilde{X}/\Gm]$ admits a 1-shifted contact structure, leading to the following definition:
\begin{definition}
    A derived \bfem{Jacobi manifold} is formalized in derived algebraic geometry as a $1$-shifted contact derived Artin stack $X$.
\end{definition}

Let $Y = \Sigma_B$ be the Betti stack of a compact, oriented 2-dimensional surface. The Betti stack is $\cO$-compact and 2-oriented ($d=2$). Applying the derived AKSZ-contact construction, the quotient mapping stack
$$
    X_{\Sigma_B} = [\Map(\Sigma_B, \widetilde{X})/\Gm]
$$
inherits a $(1-2) = -1$-shifted contact structure. 

The $(-1)$-shifted contact structure on $X_{\Sigma_B}$ provides the geometric space of fields for the derived \bfem{Jacobi Sigma Model}. In classical literature, the Jacobi sigma model action relies on restricting fields to the zero section of the symplectification. Our quotient mapping stack $X_{\Sigma_B}$ yields this geometry, mapping to the translation-invariant curves of Symplectic Field Theory.

\subsection{The Derived Courant-Jacobi Sigma Model (\texorpdfstring{$n=2, d=3$})}

 Increasing the structural shift to $n=2$ recovers the geometry of Courant-Jacobi algebroids. In graded geometry, a classical Courant algebroid corresponds to a symplectic $NQ$-manifold of degree 2. In derived algebraic geometry, this structure translates to a 2-shifted symplectic derived stack. A Courant-Jacobi algebroid introduces a modification by twisting the pairing with a line bundle $\cL$ over the base space. By \cite[Theorem 4.2]{CMM}, the graded space of fields for a Courant-Jacobi algebroid constitutes a contact $NQ$-manifold of degree 2. The symplectification of this contact $NQ$-manifold is a degree 2 symplectic $NQ$-manifold equipped with a homogeneous weight 1 
 $\mathbb{R}$-action. 
 
 Translating the correspondence above to the derived setting, the principal $\Gm$-bundle $\widetilde{X}$ over the base space $X$ is a 2-shifted symplectic derived stack equipped with a weight 1 $\Gm$-action. The $\Gm$-equivariant descent along the torsor projection down to the base stack $X \simeq [\widetilde{X}/\Gm]$ defines a 2-shifted contact structure, and hence we get:
 \begin{definition}
     A \bfem{Courant-Jacobi algebroid} is geometrized in derived algebraic geometry as a $2$-shifted contact derived stack $X$.
 \end{definition}

Let $Y = M_B$ be the Betti stack of a compact, oriented 3-dimensional manifold ($d=3$), hence $\cO$-compact and $3$-oriented. The quotient mapping stack 
$$
    X_{M_B} = [\Map(M_B, \widetilde{X})/\Gm]
$$
admits a $(2-3) = -1$-shifted contact structure. This derived space provides the moduli formulation of the 3-dimensional \bfem{Courant-Jacobi topological field theory}. The $(-1)$-shifted geometry controls the shifted BV action functional, incorporating Wess-Zumino-Witten terms through the non-triviality of the twisting line bundle $\cO(1)$ on the mapping stack.

\subsection{Derived Loop Spaces and Contact Integration (\texorpdfstring{$n=1, d=1$})}

While evaluating the mapping stack on higher-dimensional source spaces generates field theories, restricting to 1-dimensional source spaces recovers geometric integration procedures. A major application of the AKSZ formalism in Poisson geometry is the integration of Poisson manifolds to symplectic groupoids via the derived loop space. We establish the contact analogue for integrating Jacobi structures.

Let $X$ be a $1$-shifted contact derived Artin stack, representing a derived Jacobi manifold. Let $Y = S^1_B$ be the Betti stack of the circle, which is $\cO$-compact and 1-oriented ($d=1$). The derived quotient mapping stack $X_{S^1_B} = [\Map(S^1_B, \widetilde{X})/\Gm]$ admits a $(1-1) = 0$-shifted contact structure. 

This derived space provides the moduli-theoretic construction for the path space reduction of Jacobi manifolds. It evaluates the derived space of paths as a $0$-shifted contact stack, realizing the higher geometric integration of derived Jacobi structures into derived contact groupoids.

\subsection{The Contact PTVV Theorem for Perfect Complexes}

As a further application of our quotient mapping stack formalism, we establish the contact analogue of the PTVV theorem for moduli spaces of perfect complexes \cite[Theorem 0.1]{PTVV}. In the classical setting, the derived moduli stack of perfect complexes $\mathrm{Perf}$ acts as a $2$-shifted symplectic target. To transition to the contact setting, we replace this with its $2$-shifted contact prequantization: the \bfem{projectivized $2$-shifted cotangent stack} $\mathbb{P}(T^*[2]\mathrm{Perf}) := [T^*[2]\mathrm{Perf}^\circ / \Gm]$, which is analogous to the construction of $\mathbb{P}(T^*[2]BG)$ in \cite[Corollary 6.4]{IzbudakBerktav2026}.

By evaluating this target over various oriented source spaces, and carefully analyzing the necessary conditions for strong non-degeneracy, we obtain the following complete characterization.

\begin{proposition}[Contact PTVV and Calabi-Yau Spaces] \label{prop: contact Calabi-Yau}
Let $Y$ be an $\cO$-compact derived stack. The quotient mapping stack 
$$[\Map(Y, T^*[2]\mathrm{Perf}^\circ)/\Gm]$$ 
canonically admits a strict $(2-d)$-shifted contact structure via transgression if and only if $Y$ is a Calabi-Yau derived stack (or Calabi-Yau category) of dimension $d$. 

Specifically, we have the following geometric evaluations:
\begin{enumerate}\itemsep=0.1in
    \item \bfem{Calabi-Yau Case:} If $Y$ is a smooth and proper Calabi-Yau variety of dimension $d$, the quotient mapping stack evaluates to the derived moduli stack of \bfem{projective Higgs perfect complexes} on $Y$ $$\mathbb{P}(T^*[2-d]\mathrm{Perf}(Y)),$$ which admits a canonical $(2-d)$-shifted contact structure.
    
    \item \bfem{De Rham Case:} If $Y = Z_{DR}$ is the de Rham stack of a smooth and proper variety $Z$ of dimension $d$, the quotient mapping stack evaluates to the \bfem{projectivized cotangent stack of perfect complexes with flat connections} on $Z$, denoted $$\mathbb{P}(T^*[2(1-d)]\mathrm{Perf}_{DR}(Z)),$$ which admits a canonical $2(1-d)$-shifted contact structure.
    
    \item \bfem{Betti Case:} If $Y = M_B$ is the Betti stack associated with a compact oriented topological manifold $M$ of dimension $d$, the quotient mapping stack evaluates to the \bfem{projectivized cotangent stack of perfect complexes of local systems} on $M$, denoted $$\mathbb{P}(T^*[2-d]\mathrm{Loc}(M, \mathrm{Perf})),$$ which admits a canonical $(2-d)$-shifted contact structure.
\end{enumerate}
\end{proposition}

\begin{proof}
Let $\widetilde{\mathcal{F}} = T^*[2]\mathrm{Perf}^\circ$. By Theorem \ref{thm:A} and Theorem \ref{thm:quotient_contact}, the quotient stack $$[\Map(Y, \widetilde{\mathcal{F}})/\Gm]$$ admits a strict $(2-d)$-shifted contact structure if and only if its derived symplectification, given by the unmodified mapping space $\Map(Y, \widetilde{\mathcal{F}})$, carries a weight $1$, strictly non-degenerate $(2-d)$-shifted symplectic structure.

The AKSZ transgression mechanism dictates that pushing forward the $2$-shifted symplectic form $\omega_{\widetilde{\mathcal{F}}}$ along the evaluation map $ev \colon Y \times \Map(Y, \widetilde{\mathcal{F}}) \to \widetilde{\mathcal{F}}$ requires a continuous trace functional $\int_{[Y]} \colon \mathbb{R}\Gamma(Y, \cO_Y) \to \K[-d]$. The existence of this functional implies that $Y$ must be an $\cO$-compact derived stack equipped with an $\cO$-orientation of dimension $d$.

Let $f \colon Y \to \widetilde{\mathcal{F}}$ be a given map. The tangent complex of the derived mapping stack at $f$ evaluates to $\TT_f \simeq \mathbb{R}\Gamma(Y, f^*\TT_{\widetilde{\mathcal{F}}})$. Because $\widetilde{\mathcal{F}}$ is strictly $2$-shifted symplectic, the intrinsic equivalence $\TT_{\widetilde{\mathcal{F}}} \xrightarrow{\sim} \LL_{\widetilde{\mathcal{F}}}[2]$ induces a pairing on the tangent complex of the mapping space via the composition
\[ \mathbb{R}\Gamma(Y, f^*\TT_{\widetilde{\mathcal{F}}}) \otimes^\LL \mathbb{R}\Gamma(Y, f^*\TT_{\widetilde{\mathcal{F}}}) \xrightarrow{\sim} \mathbb{R}\Gamma(Y, f^*(\TT_{\widetilde{\mathcal{F}}} \otimes^\LL \TT_{\widetilde{\mathcal{F}}})) \xrightarrow{\omega_{\widetilde{\mathcal{F}}}} \mathbb{R}\Gamma(Y, \cO_Y[2]) \xrightarrow{\int_{[Y]}} \K[2-d]. \]
The transgressed $2$-form is strictly non-degenerate if and only if this pairing is a quasi-isomorphism. For a map $f$ factoring through the zero section of the cotangent bundle, corresponding to a perfect complex $E \in \mathrm{Perf}(Y)$, the pullback $f^*\TT_{\widetilde{\mathcal{F}}}$ splits into a direct sum of the tangent and shifted cotangent complexes of $\mathrm{Perf}$ evaluated at $E$. The non-degeneracy of the canonical hyperbolic pairing on $f^*\TT_{\widetilde{\mathcal{F}}}$ restricts the mapping space non-degeneracy to the condition that the trace map induces a perfect pairing on the derived endomorphism algebra:
\[ \mathbb{R}\Hom_Y(E, E) \otimes^\LL \mathbb{R}\Hom_Y(E, E) \xrightarrow{\text{comp}} \mathbb{R}\Gamma(Y, \cO_Y) \xrightarrow{\int_{[Y]}} \K[-d]. \]

This requires an equivalence $\mathbb{R}\Hom_Y(E, E) \xrightarrow{\sim} \mathbb{R}\Hom_Y(E, E)^\vee[-d]$ for all $E \in \mathrm{Perf}(Y)$, which holds if and only if $Y$ satisfies Serre duality of dimension $d$ with a trivial dualizing sheaf ($\omega_Y \simeq \cO_Y$). In derived algebraic geometry, this is precisely the defining condition for $Y$ to be a Calabi-Yau derived stack of dimension $d$. Under this assumption, Serre duality yields the equivalence $\Map(Y, \widetilde{\mathcal{F}}) \simeq T^*[2-d]\mathrm{Perf}(Y)^\circ$, and the subsequent $\Gm$-equivariant quotient isolates the weight $1$ structure to yield the canonical $(2-d)$-shifted contact structure on the projectivization.

Conversely, if $Y$ is a smooth and proper variety but lacks a trivial dualizing sheaf ($\omega_Y \not\simeq \cO_Y$), the fundamental integration pairing possesses a non-trivial kernel. The ensuing degeneracy in the transgressed form ensures the quotient mapping stack admits only a weak contact structure, as characterized by Theorem \ref{prop:weak_contact}.

Finally, the required non-degeneracy criteria are intrinsically satisfied in the de Rham and Betti contexts via their respective canonical dualities:
\medskip

{\textbf{De Rham Case:}} For the de Rham stack $Y = Z_{DR}$ of a smooth and proper $d$-dimensional variety $Z$, algebraic Poincar\'e duality in de Rham cohomology provides the necessary $2d$-orientation. Transgression yields a strict $2(1-d)$-shifted symplectic structure on the unmodified mapping stack. Passing to the $\Gm$-quotient thereby establishes the canonical $2(1-d)$-shifted contact structure on the projectivized cotangent stack of flat perfect complexes.
\medskip

{\textbf{Betti Case:}} For the Betti stack $Y = M_B$ of a compact oriented $d$-manifold $M$, topological Poincar\'e duality in Betti cohomology supplies a non-degenerate $d$-orientation. This guarantees the perfection of the pairing, yielding a strictly non-degenerate $(2-d)$-shifted symplectic form on the mapping space. The $\Gm$-equivariant quotient subsequently descends this geometry to a strict $(2-d)$-shifted contact structure on the projectivization of local systems. \qedhere
\end{proof}

\begin{remark}[Symplectic Case]
For an $\cO$-compact derived stack $Y$, the standard derived mapping stack $\Map(Y, \mathrm{Perf})$ canonically inherits a strict $(2-d)$-shifted symplectic structure via the AKSZ transgression mechanism \bfem{if and only if} $Y$ is a Calabi-Yau space (or category) of dimension $d$. For spaces that are not Calabi-Yau, the canonical constructions yield degenerate forms. (To study non-Calabi-Yau spaces, derived geometry usually shifts to looking at shifted Poisson structures on the mapping stack, or mapping stacks relative to a boundary or anti-canonical divisor).
\end{remark}

\subsection{Linearized Contact Extended Topological Field Theories}
Beyond the specific low-dimensional examples evaluated above, the quotient mapping stack construction assigns geometric data to topological cobordisms of arbitrary dimension. 

Recall from Section \ref{sec:fillings} that derived fillings induce Legendrian boundary conditions, and that the topological gluing of cobordisms evaluates to derived Legendrian intersections. By evaluating the quotient mapping stack across the symmetric monoidal $(\infty, d)$-category of oriented cobordisms, this generates a geometric Extended Topological Field Theory (ETFT) taking values in the non-linear higher category of Legendrian correspondences \cite{IzbudakBerktav2}:
\[ \mathcal{Z}_c \colon Bord_d \to Leg_{n-d+1}. \]

To define numerical invariants and linearize this topological field theory, one requires a contact analogue of the BBDJS perverse sheaves on the derived critical loci \cite{BBDJS}. This is supplied by two companion papers of the first author. The first establishes the Equivariant Contact Darboux Theorem, verifying that $(-1)$-shifted contact stacks are locally geometric quotients of derived discriminant loci \cite[Theorem 3.8]{Izbudak_Darboux}. Building on that geometry, the second constructs canonical monodromic $\ell$-adic perverse sheaves \cite[Theorem A]{Izbudak_Monodromic} and uses them to linearize the non-linear Legendrian $2$-category \cite[Theorem C]{Izbudak_Monodromic}. This yields a $2$-functor
\[ F \colon Leg_0^{\mathrm{pr}} \longrightarrow LLeg_0 , \]
monoidal for the contact product and taking values in the perversely categorified Legendrian category. Here $Leg_0^{\mathrm{pr}} \subseteq Leg_0$ is the sub-$2$-category on all objects whose $1$- and $2$-morphisms are proper and carry graded orientations, and the construction of $F$ is conditional on the monodromic refinement of the symplectic Joyce conjecture assumed in \cite[Setup 5.1]{Izbudak_Monodromic}.

By restricting our quotient mapping stack to $n = d - 1$, we can post-compose the geometric ETFT with this perverse linearization, provided the Legendrian spans it produces are proper and graded oriented, so that they lie in $Leg_0^{\mathrm{pr}}$.

\begin{theorem}[Cohomological Contact ETFT]\label{thm:linear_etft}
Assume the hypotheses of \cite[Setup 5.1]{Izbudak_Monodromic} and suppose $\mathcal{Z}_c$ takes values in $Leg_0^{\mathrm{pr}}$. Then the composition $\mathcal{Z}^{lin}_c := F \circ \mathcal{Z}_c$ defines a Cohomological Contact Extended Topological Field Theory:
\[ \mathcal{Z}^{lin}_c \colon Bord_d \to LLeg_0. \]
\end{theorem}
This evaluates the topological gluing of moduli problems on $d$-manifolds as the cohomological composition of local $\ell$-adic vanishing cycles on the derived contact boundaries, connecting the topological field theory of mapping stacks to the microlocal sheaf theory of contact spaces.

\section*{Acknowledgements}
We would like to thank the Higher Structures group at Middle East Technical University (METU) for their valuable discussions, insights, and continuous support throughout this work.

\bibliographystyle{amsalpha}
\bibliography{refss}

\end{document}